\documentclass[12pt]{article}
\usepackage{array,amsmath}
\usepackage{amssymb}
\usepackage{graphicx,subfigure}

\makeatletter
 \usepackage{fullpage} 
 \usepackage{amsthm}

\date{}

\usepackage{fullpage}

\makeatother

\begin{document}

\newtheorem{thm}{Theorem}[section]
\newtheorem{lem}[thm]{Lemma}
\newtheorem{prop}[thm]{Proposition}
\newtheorem{cor}[thm]{Corollary}
\newtheorem{con}[thm]{Conjecture}
\newtheorem{cons}[thm]{Construction}
\newtheorem{claim}[thm]{Claim}
\newtheorem{obs}[thm]{Observation}
\newtheorem{rmk}[thm]{Remark}
\newtheorem{defn}[thm]{Definition}
\newtheorem{example}[thm]{Example}
\newcommand{\di}{\displaystyle}
\def\dfc{\mathrm{def}}
\def\cF{{\cal F}}
\def\cH{{\cal H}}
\def\cT{{\cal T}}
\def\cM{{\cal M}}
\def\cA{{\cal A}}
\def\cB{{\cal B}}
\def\cG{{\cal G}}
\def\ap{\alpha'}
\def\Frk{F_k^{2r+1}}
\def\nul{\varnothing} 
\def\st{\colon\,}   
\def\MAP#1#2#3{#1\colon\,#2\to#3}
\def\VEC#1#2#3{#1_{#2},\ldots,#1_{#3}}
\def\VECOP#1#2#3#4{#1_{#2}#4\cdots #4 #1_{#3}}
\def\SE#1#2#3{\sum_{#1=#2}^{#3}}  \def\SGE#1#2{\sum_{#1\ge#2}}
\def\PE#1#2#3{\prod_{#1=#2}^{#3}} \def\PGE#1#2{\prod_{#1\ge#2}}
\def\UE#1#2#3{\bigcup_{#1=#2}^{#3}}
\def\FR#1#2{\frac{#1}{#2}}
\def\FL#1{\left\lfloor{#1}\right\rfloor} 
\def\CL#1{\left\lceil{#1}\right\rceil}

\title{Edge-connectivity in regular multigraphs from eigenvalues}
\author{Suil O\thanks{Department of Mathematics and Statistics, Georgia State University, Atlanta, GA, 30303, suilo@gsu.edu.}
}

\maketitle

\begin{abstract}
Let $G$ be a $d$-regular multigraph, and let $\lambda_2(G)$ be the second largest eigenvalue of $G$.
In this paper, we prove that if $\lambda_2(G) < \frac{d-1+\sqrt{9d^2-10d+17}}4$,
then $G$ is 2-edge-connected. 
Furthermore, for $t\ge2$ we show that $G$ is $(t+1)$-edge-connected when $\lambda_2(G)<d-t$, and 
in fact when $\lambda_2(G)<d-t+1$ if $t$ is odd.
\end{abstract}

\section {Introduction}
A {\it simple} graph is a graph without loops or multiple edges. In this paper, a {\it multigraph} is a graph
that can have multiple edges but does not contain loops. A simple graph or multigraph $G$
is {\it $k$-connected} if $G$ has more than $k$ vertices and every subgraph obtained by deleting fewer than $k$ vertices is connected;
the {\it connectivity} of $G$, written $\kappa(G)$, is the maximum $k$ such that $G$ is $k$-connected.
The {\it adjacency matrix} $A(G)$ of $G$ is the $n$-by-$n$ matrix in which the entry $a_{i,j}$ is the number of edges in $G$ with endpoints $\{v_i, v_j\}$,
where $V(G)=\{v_1, ... , v_n\}$. The {\it eigenvalues} of $G$ are the eigenvalues of its adjacency matrix $A(G)$.
Let $\lambda_1(G), ... , \lambda_n(G)$ be its eigenvalues indexed in nonincreasing order.
The {\it Laplacian matrix} of $G$ is $D(G)-A(G)$, where $D(G)$ is the diagonal matrix of degrees.
Let $\mu_1(G), ... , \mu_n(G)$ be its eigenvalues indexed in nondecreasing order.
Note that if $G$ is a $d$-regular graph, then $\lambda_i(G)=d-\mu_i(G)$ for $1 \le i \le n$.

A lot of research in graph theory over the last 40 years was stimulated by a classical result of Fiedler~\cite{F}, stating that 
\begin{equation}\label{fieeq}
\kappa(G) \ge \mu_2(G)\end{equation}
for a non-complete simple graph $G$. In 2002, Kirkland, Molitierno, 
Neumann, and Shader~\cite{KMNS} characterized when equality in the inequality~(\ref{fieeq}) holds.


A simple graph or multigraph $G$ is {\it $t$-edge-connected} if every subgraph obtained by deleting fewer than $t$ edges is connected;
the {\it edge-connectivity} of $G$, written $\kappa'(G)$, is the maximum $t$ such that $G$ is $t$-edge-connected.
Note that $\kappa(G) \le \kappa'(G)$.
Chandran~\cite{C1} proved that if $G$ is an $n$-vertex $d$-regular simple graph with \begin{equation}\label{chan}
             \lambda_2(G) < d-1 -\frac{d}{n-d},
            \end{equation} then $\kappa'(G)=d$ and if $|[S,\overline{S}]|=d$, where for vertex sets $S$ and $T$, $[S,T]$ is the set of edges from $S$ to $T$,
then $S$ equals a single vertex or $V(G-v)$ for a vertex $v \in V(G)$. 
Krivelevich and Sudakov~\cite{KS} slightly improved Chandran's result and showed that
if $G$ is a $d$-regular simple graph with $\lambda_2(G) \le d-2$, then $\kappa'(G)=d$. In 2010, Cioab{\v a}~\cite{C2} proved that if $$\lambda_2(G) < d - \frac{2t}{d+1},$$
then $\kappa'(G) \ge t+1$. When $t$  equals $1$ or $2$, he proved stronger results .

\begin{thm}\label{sebimain1}{\rm \cite{C2}}
Let $d$ be an odd integer at least 3 and let $\pi(d)$ be the largest root
of $x^3-(d-3)x^2-(3d-2)x-2=0$. If $G$ is a $d$-regular simple graph such that
$\lambda_2(G)<\pi(d)$, then $\kappa'(G) \ge 2$.
\end{thm}

This result is best possible in a sense that there exists a $d$-regular simple graph $H$ with $\lambda_2(H)=\pi(d)$ and with $\kappa'(H)=1$ (See~\cite{C2}).

\begin{thm}\label{sebimain2}{\rm \cite{C2}}
If $G$ is a $d$-regular simple graph such that $\lambda_2(G) < \frac{d-3+\sqrt{(d+3)^2-16}}{2}$,
then $\kappa'(G) \ge 3$.
\end{thm}

This result is also best possible in a sense that there exists a $d$-regular simple graph $K$ with $\lambda_2(K)= \frac{d-3+\sqrt{(d+3)^2-16}}{2}$ and with $\kappa'(K)=2$ (See~\cite{C2}). 
To guarantee higher edge-connectivity, the author conjectured in his thesis~[8].

\begin{con}  {\rm ~\cite{O}} \label{conj}
Let $\rho(d,t)=\begin{cases} 
\frac{d-4 + \sqrt{(d+4)^2-8t}}{2},  \text{ when } t \text{ is odd } \\
\frac{d-3 + \sqrt{(d+3)^2-8t}}{2}, \text{ when } t \text{ is even }.\end{cases}$ \\
For $t \ge 3$, if $G$ is a $d$-regular simple graph such that $\lambda_2(G) < \rho(d,t)$,
then $\kappa'(G) \ge t+1$.
\end{con}

In fact, the author~\cite{O} proved that for a $d$-regular simple graph $G$ with $\kappa'(G) \le t$,
if there exists a vertex subset $S \subseteq V(G)$ with $|[S,\overline{S}]|=\kappa'(G)$ and both $|S|$ and $|\overline{S}|$
are at least $d+4$ when $d$ is odd and at least $d+3$ when $d$ is even, then $\lambda_2(G) \ge \rho(d,t)$.
This gives a partial positive answer to Conjecture 1.3. For each positive integer $t \ge 2$, 
the author constructed a $t$-connected $d$-regular simple graph $G$ with $\lambda_2(G)=\rho(d,t)$,
so the bound on the second largest eigenvalue is best possible if Conjecture~\ref{conj} is true.

In this paper, we extend some of these results to multigraphs. The main results of this paper are Theorem~\ref{main1} and Theorem~\ref{main2}.
In fact, we obtain a best possible condition on the second largest eigenvalue $\lambda_2(G)$ for a $d$-regular multigraph $G$ 
to guarantee that $\kappa'(G) \ge t+1$ for any positive integer $t$. We prove separately when $t=1$ in Section 3 and when $t \ge 2$ in Section 4.

\begin{thm}\label{main1}
If $G$ is a connected $d$-regular multigraph with $\lambda_2(G) < \frac{d-1+\sqrt{9d^2-10d+17}}4$, then $\kappa'(G) \ge 2$.
\end{thm}

In Section 2, for any positive integer $d \ge 3$,
 we construct an example of a $d$-regular multigraph $H_{d,1}$ having $\lambda_2(H_{d,1})=\frac{d-1+\sqrt{9d^2-10d+17}}4$
 and  $\kappa'(H_{d,1})=1$. Our construction shows that Theorem~\ref{main1} is best possible.
To guarantee higher edge-connectivity, we need a smaller upper bound on $\lambda_2$ in terms of $d$ and $t$.

\begin{thm}\label{main2}
For $t \ge 2$, if $G$ is a connected $d$-regular multigraph with $\lambda_2(G) < d-t$,
then $\kappa'(G) \ge t+1$. Furthermore, if $t$ is odd and $G$ is a connected $d$-regular multigraph with $\lambda_2(G) < d-t+1$, then $\kappa'(G) \ge t+1$.
\end{thm}

We can rephrase Theorem~\ref{main2} as follows: if $G$ is a connected $d$-regular multigraph with
$\lambda_2(G) < d - 2l$ for $l \ge 1$, then $\kappa'(G) \ge 2l+2$. Our result implies that as the upper bound
on $\lambda_2(G)$ for a $d$-regular graph $G$ decreases by 2, the lower bound on $\kappa'(G)$ increases by 2.
However, decreasing the bound on $\lambda_2(G)$ by 1 does not increase the lower bound on $\kappa'(G)$ by 1;
this is a sense in which our result is sharp.  For example, $\lambda_2(G)<d-2$ implies $\kappa'(G)\ge 4$, and
$\lambda_2(G)<d-4$ implies $\kappa'(G)\ge 6$, but $\lambda_2(G)<d-3$ does not imply $\kappa'(G)\ge 5$. In Section 2,
for any positive integer $d \ge 3$ and for even positive integer $t$, we construct a $d$-regular multigraph
$H_{d,t}$ with $\kappa'(H_{d,t})=t$ and $\lambda_2(H_{d,t})=d-t$.  Thus Theorem~\ref{main2} is best possible for all $t \ge 2$ 
regardless of whether $t$ is odd or even.

For undefined terms, see West~\cite{W} or Godsil and Royle~\cite{GR}.

\section{The Construction}
Cioab{\v a}~\cite{C2} presented examples for sharpness in the upper bound on $\lambda_2(G)$
for a $d$-regular simple graph $G$ to guarantee that $G$ is either 2-edge-connected or 3-edge-connected.
In this section, we present a smallest $d$-regular multigraph $H_{d,1}$ with $\kappa'(H_{d,1})=1$ and with $\lambda_2(H_{d,1})=\frac{d-1+\sqrt{9d^2-10d+17}}4$, and a smallest $d$-regular multigraph $H_{d,t}$ with $\kappa'(H_{d,t})=t$
and with $\lambda_2(H_{d,t})=d-t$ for every even positive integer $t \ge 2$.

\begin{obs} \label{obs}
For a positive integer $d \ge 3$,
there exists only one multigraph with three vertices such that every vertex has degree $d$,
except only one vertex with degree $d-1$.
\end{obs}
\begin{proof}
Let $G$ be a multigraph with three vertices, say $v_1, v_2$, and $v_3$ such that $v_1$ and $v_2$ have degree $d$,
and $v_3$ has degree $d-1$. Let $a, b, $ and $c$ be the number of edges between $v_1$ and $v_2$,
between $v_2$ and $v_3$, and between $v_3$ and $v_1$, respectively.
Since $v_1$ and $v_2$ have degree $d$, and $v_3$ has degree $d-1$,
we have $a+c=a+b=d$, $c+b=d-1$, and $2(a+b+c)=3d-1$. Thus, we have $a=\frac{d+1}2$ and $b=c=\frac{d-1}2$,
which gives the desired result.
\end{proof}

\begin{cons}{\rm Let $B_{d,1}$ be the 3-vertex multigraph guaranteed by Observation~\ref{obs}.
Let $H_{d,1}$ be the graph obtained from two copies of $B_{d,1}$ by adding one edge between
the two vertices with degree $d-1$ in the copies. 
Note that $H_{d,1}$ is the smallest $d$-regular multigraph with $\kappa'(H_{d,1})=1$.

Let $t$ be an even positive integer less than $d-1$, and  
let $B_{d,t}$ be the multigraph with two vertices, say $x$ and $y$, such that every vertex has degree $d-\frac t2$.
Let $H_{d,t}$ be the graph obtained from two copies of $B_{d,t}$ by adding $\frac t2$ edges between
two $x$ and between two $y$ in the copies. Note that $H_{d,t}$ is a smallest $d$-regular multigraph with $\kappa'(H_{d,t})=t$.
}
\end{cons}

When $d=3$, see Figure 1 and 2 for $t=1$ and for $t=2$, respectively.

Now, we determine the second largest eigenvalue of $H_{d,t}$.
First, we introduce some definitions.
Consider a partition $V(G) = V_1 \cup \cdots V_s$ of the vertex set of $G$
into $s$ non-empty subsets. For $1 \le i,j \le s$, 
let $b_{i,j}$ denote the average number of neighbours in $V_j$ of the vertices in $V_i$. 
The quotient matrix of this partition is the $s\times s$ matrix 
whose $(i, j)$-th entry equals $b_{i,j}$. 

\begin{thm} {\rm (see Corollary 2.5.4 in \cite{BH}, Lemma 9.6.1 in \cite{GR})}\label{GR1}
The eigenvalues of the quotient matrix interlace the eigenvalues of $G$. 
\end{thm}

This partition is {\it equitable} if for each $1\le i,j \le s$, 
any vertex $v \in V_i$ has exactly $b_{i,j}$ neighbours in $V_j$ . 

\begin{thm} {\rm (see Lemma 2.3.1 in \cite{BH}, Theorem 9.3.3 in \cite{GR})}\label{GR2}
The set of the eigenvalues of $G$ includes the eigenvalues of the quotient matrix for an equitable partition of $G$.
\end{thm}

\begin{figure}
\begin{center}
\includegraphics[height=2cm]{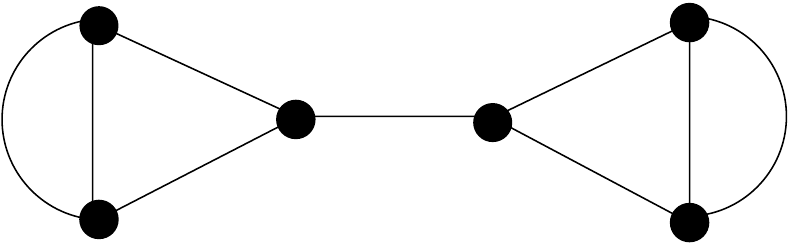}
\end{center}
\caption{A cubic multigraph with $\lambda_2= \frac{1+\sqrt{17}}2$ and $\kappa'=1$}
\end{figure}

\begin{figure}
\begin{center}
\includegraphics[height=2cm]{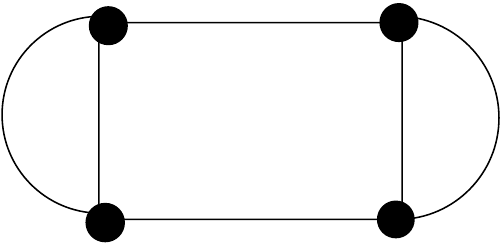}
\end{center}
\caption{A cubic multigraph with $\lambda_2= 1$ and $\kappa'=2$}
\end{figure}

\begin{thm}{\rm (see Proposition 3.4.1 in \cite{BH}, Theorem 8.8.2 in \cite{GR})}\label{L}
A graph is bipartite if and only if its spectrum is symmetric about the origin.
\end{thm}

\begin{thm}
The second largest eigenvalues of $H_{d,1}$ and $H_{d,t}$ are $\frac{d-1+\sqrt{9d^2-10d+17}}4$ and $d-t$, respectively.
\end{thm}
\begin{proof}
Partition the vertex set of $H_{d,t}$ into two parts: two copies of $V(B_{d,t})$ in $H_{d,t}$.
Note that this partition is equitable, and its quotient matrix is
$$A=\begin{pmatrix}
    d-\frac t2 & \frac t2 \\
    \frac t2  & d - \frac t2
\end{pmatrix}.$$
The characteristic polynomial of the matrix is $(x-d)(x-d+t)$. By Theorem~\ref{GR2}, the numbers $d$ and $d-t$ are eigenvalues of $H_{d,t}$, and by Theorem~\ref{L}, the eigenvalues of $H_{d,t}$ are $-d, -(d-t)$, $d-t$, and $d$.
Thus $\lambda_2(H_{d,t})=d-t$.

The adjacency matrix of $H_{d,1}$ is 

$$\begin{pmatrix}
    0 & \frac{d+1}{2} & \frac{d-1}{2} & 0 & 0 & 0 \\
    \frac{d+1}2 & 0 & \frac{d-1}{2} &  0 & 0 & 0  \\
    \frac{d-1}2 & \frac{d-1}2 & 0 & 1 & 0 & 0  \\
 0 & 0 & 1 & 0 &  \frac{d-1}2 & \frac{d-1}2\\
0 & 0 & 0 & \frac{d-1}2 & 0 & \frac{d+1}{2}\\
 0 & 0 & 0 & \frac{d-1}{2}  & \frac{d+1}{2} & 0 \\
   
\end{pmatrix}.$$
The characteristic polynomial of this matrix is  $$(x-d)\left\{x^2-(\frac{d-1}2)x-(\frac{d^2-d+2}2)\right\}(x+\frac{d-3}2)(x+\frac{d+1}2)^2,$$
so the eigenvalues of $H_{d,1}$ are $d, \frac{d-1\pm \sqrt{9d^2-10d+17}}4, -\frac{d+1}2, -\frac{d+1}2$, and  $-\frac{d-3}2$. Thus, we have $\lambda_2(H_{d,1})=\frac{d-1+\sqrt{9d^2-10d+17}}4$.\\
\end{proof}

\section{Proof of Theorem~\ref{main1}}

In this section, we prove Theorem~\ref{main1} by finding appropriate partitions of a $d$-regular multigraph $G$ with $\kappa'(G)=1$
such that the quotient matrices of these partitions have their second largest eigenvalues greater than equal to $\frac{d-1+\sqrt{9d^2-10d+17}}2$.
By Theorem~\ref{GR1}, we have $\lambda_2(G) \ge \frac{d-1+\sqrt{9d^2-10d+17}}2$ . Now, we prove Theorem~\ref{main1}.

\medskip
\noindent
{\it Proof of Theorem~\ref{main1}}.
Assume to the contrary that $\kappa'(G) = 1$. Then there exists a vertex subset $S$ such that $|[S,\overline{S}]|=1$.
Let $a=|S|$ and let $b=|\overline{S}|$, so we have $a+b=n$. We may assume that $a \le b$.
Note that  $d$, $a$, and $b$ must be odd by degree-sum formula, and $a$ is at least 3. The quotient matrix of the partition $S$ and $\overline{S}$ is 

$$A=\begin{pmatrix}
    d-\frac{1}{a} & \frac{1}{a} \\
    \frac{1}{b}  & d - \frac{1}{b}
\end{pmatrix}.$$

The characteristic polynomial of this matrix equals 
$$\left(x-d+\frac{1}{a}\right)\left(x-d+\frac{1}{b}\right)-\frac{1}{ab}$$,
$$=(x-d)^2+\left(\frac{1}{a}+\frac{1}{b}\right)(x-d)=(x-d)\left(x-d+\frac{1}{a}+\frac{1}{b}\right).$$
Thus the eigenvalues of the matrix $A$ are $d$ and $d-\frac{1}{a}-\frac{1}{b}$.

\medskip
\noindent
{\it Case 1. $a \ge 5$.}
By Theorem~\ref{GR1}, we have

\begin{equation}\label{keyequ} \lambda_2(G) \ge
d- \frac{1}{a}-\frac{1}{b} \ge d - \frac{1}{5} - \frac 15 = d- \frac {2}{5}. \end{equation}
Since $(3d-\frac{3}{5})^2 -(9d^2-10d+17) > \frac{32}5 d - 17 > 0$ for $d \ge 3$,
we have $$d - \frac{2}{5} -  \frac{d-1+\sqrt{9d^2-10d+17}}4 = \frac {3d- \frac{3}{5} - \sqrt{9d^2-10d+17}}4 > 0.$$

\medskip
\noindent
{\it Case 2. $a=3$ and $b \ge 11$.}
By Theorem~\ref{GR1}, we have

\begin{equation}\label{keyequ} \lambda_2(G) \ge
d- \frac{1}{a}-\frac{1}{b} \ge d - \frac{1}{3} - \frac 1{11} = d- \frac {14}{33}. \end{equation}
Since $(3d-\frac{23}{33})^2 -(9d^2-10d+17) > \frac{64}{11} d - 17 > 0$ for $d \ge 3$,
we have $$d - \frac{14}{33} -  \frac{d-1+\sqrt{9d^2-10d+17}}4 = \frac {3d- \frac{23}{33} - \sqrt{9d^2-10d+17}}4 > 0.$$

\medskip
\noindent
{\it Case 3. $a=3$ and $5 \le b \le 9$}. Since $a=3$ and there exists only one edge between $S$ and $\overline{S}$, we have $G[S]=B_{d,1}$
by Observation~\ref{obs}. Thus there exists only one vertex, say $x$,  in $\overline{S}$ with a neighbor in $S$.
Partition the vertex set of $G$ into three parts: $V(B_{d,1})$, $\{x\}$, and $V(G)-V(B_{d,1})-\{x\}$.
The quotient matrix of this partition is
$$A=\begin{pmatrix}
    d-\frac{1}{3} & \frac{1}{3} & 0 \\
    1  &  0  & d-1 \\
    0 & c & d-c
\end{pmatrix}.$$
where $c=\frac{d-1}{b-1}$. The characteristic polynomial of the matrix is
$$(x-d+\frac 13)\left\{x(x-d+c)+c(1-d)\right\} -\frac 13(x-d+c)$$
$$=(x-d+ \frac 13)\{x^2-(d-c)x+c-cd\}-\frac 13(x-d+c)$$
$$=(x-d)\{x^2+(c-d)x+c-cd\} + \frac 13 \{x^2+(c-d)x+c-cd-x+d-c\}$$
$$=(x-d)\{x^2+(c-d)x+c-cd\}+\frac 13\{x^2+(c-d-1)x+(1-c)d\}$$
$$=(x-d)\{x^2+(c-d)x+c-cd\}+\frac 13(x-d)(x+c-1)$$
$$=(x-d)\{x^2+(c-d)x+c-cd + \frac 13(x+c-1)\}$$
$$=(x-d)\{x^2+(c-d+\frac 13)x + \frac 43 c -cd -\frac 13\}$$

The second largest root of the polynomial is
$$\frac{d - c - \frac 13 +\sqrt{(d-c-\frac 13)^2-\frac{16}3c+4cd+\frac 43}}2.$$

Now, it suffices to show that for $b \in \{5, 7, 9\}$, we have $$\frac{d - c - \frac 13 +\sqrt{(d-c-\frac 13)^2-\frac{16}3c+4cd+\frac 43}}2 ~~> ~~\frac{d-1+\sqrt{9d^2-10d+17}}4.$$

{\it Subcase 3-1. $b= 9$.}  By replacing $c$ with $\frac{d-1}8$, we have 
$$\frac{d - \frac{d-1}8 - \frac 13 +\sqrt{(d-\frac{d-1}8-\frac 13)^2-\frac{16}3\frac{d-1}8+4\frac{d-1}8d+\frac 43}}2
=\frac{21d-5+\sqrt{729d^2-882d+1177}}{48}.$$
By subtracting $\frac{d-1+\sqrt{9d^2-10d+17}}4$ from $\frac{21d-5+\sqrt{729d^2-882d+1177}}{48}$, we have
$$\frac{9d+7+\sqrt{729d^2-882d+1177}-12\sqrt{9d^2-10d+17}}{48}.$$
Since both $9d+7+\sqrt{729d^2-882d+1177}$ and $12\sqrt{9d^2-10d+17}$ are positive,
it suffices to show that $0 < \left(9d+7+\sqrt{729d^2-882d+1177}\right)^2 - \left(12\sqrt{9d^2-10d+17}\right)^2$
$$=2(9d+7)\sqrt{729d^2-882d+1177} - (486d^2 -684d+790).$$ 
Since both $2(9d+7)\sqrt{729d^2-882d+1177}$ and $486d^2-684d+790$ are positive for $d \ge 1$,
it suffices to show that $0 < \left(2(9d+7)\sqrt{729d^2-882d+1177}\right)^2 - \left(486d^2-684d+790\right)^2$
$$=236196d^4+81648d^3+79704d^2+420336d+230692$$ $$- \left(236196d^4-664848d^3+1235736d^2-1080720d+624100\right)$$ 
$$=746496d^3-1156032d^2+1501056d-393408.$$
By comparing each coefficient of $d^i$ for $i \in \{0,1,2,3\}$, we have this inequality
 $$746496d^3-1156032d^2+1501056d-393408 > 746496d^3 - 1194393.6d^2 + 1492992d-447897.6$$
$$=746496(d^3-1.6d^2+2d-0.6)=746496\{d^2(d-1.6)+2(d-1.6) + 2.6\}$$ $$ >746496(d^2+2)(d-1.6)>0$$
for $d \ge 3$.

{\it Subcase 3-2. $b=7$}. By plugging 7 into $b$, we have 
$$\frac{d - \frac{d-1}6 - \frac 13 +\sqrt{(d-\frac{d-1}6-\frac 13)^2-\frac{16}3\frac{d-1}6+4\frac{d-1}6d+\frac 43}}2=\frac{5d-1+\sqrt{49d^2-66d+81}}{12}.$$
By subtracting $\frac{d-1+\sqrt{9d^2-10d+17}}4$ from $\frac{5d-1+\sqrt{49d^2-66d+81}}{12}$, we have
$$\frac{2d+2+\sqrt{49d^2-66d+81}-3\sqrt{9d^2-10d+17}}{12}.$$
Since both $2d+2+\sqrt{49d^2-66d+81}$ and $3\sqrt{9d^2-10d+17}$ are positive,
it suffices to show that $0 < \left(2d+2+\sqrt{49d^2-66d+81}\right)^2 - \left(3\sqrt{9d^2-10d+17}\right)^2$
$$=4(d+1)\sqrt{49d^2-66d+81} - (28d^2 -32d+68).$$ 
Since both $4(d+1)\sqrt{49d^2-66d+81}$ and $28d^2-32d+68$ are positive for $d \ge 1$,
it suffices to show that $0 < \left(4(d+1)\sqrt{49d^2-66d+81}\right)^2 - \left(28d^2-32d+68\right)^2$
$$=784d^4+512d^3-32d^2+1536d+1296- \left(784d^4-1792d^3+4832d^2-4352d+4624\right)$$ 
$$=2304d^3-4864d^2+5888d-3328.$$
By comparing each coefficient of $d^i$ for $i \in \{0,1,2,3\}$, we have this inequality
 $$2304d^3-4864d^2+5888d-3328 > 2304d^3 - 5068.8d^2 + 4608d-10137.6$$
$$=2304(d^3-2.2d^2+2d-4.4)=2304\{d^2(d-2.2)+2(d-2.2)\}$$ $$ >746496(d^2+2)(d-2.2)>0$$
for $d \ge 3$.

{\it Subcase 3-3. $b=5$}. If we plug 5 into $b$, then we have 
$$\frac{d - \frac{d-1}4 - \frac 13 +\sqrt{(d-\frac{d-1}4-\frac 13)^2-\frac{16}3\frac{d-1}4+4\frac{d-1}4d+\frac 43}}2=\frac{9d-1+\sqrt{225d^2-354d+385}}{24}.$$
By subtracting $\frac{d-1+\sqrt{9d^2-10d+17}}4$ from $\frac{9d-1+\sqrt{225d^2-354d+385}}{24}$, we have
$$\frac{3d+5+\sqrt{225d^2-354d+385}-6\sqrt{9d^2-10d+17}}{24}.$$
Since both $3d+5+\sqrt{225d^2-354d+385}$ and $6\sqrt{9d^2-10d+17}$ are positive,
it suffices to show that $0 < \left(3d+5+\sqrt{225d^2-354d+385}\right)^2 - \left(6\sqrt{9d^2-10d+17}\right)^2$
$$=2(3d+5)\sqrt{225d^2-354d+385} - (90d^2 -36d+202).$$ 
Since both $2(3d+5)\sqrt{225d^2-354d+385}$ and $90d^2-36d+202$ are positive for $d \ge 1$,
it suffices to show that $0 < \left(2(3d+5)\sqrt{225d^2-354d+385}\right)^2 - \left(90d^2-36d+202\right)^2$
$$=8100d^4+14256d^3-6120d^2+10800d+38500- \left(8100d^4-6480d^3+37656d^2-14544d+40804\right)$$ 
$$=20736d^3-43776d^2+25344d-2304.$$
By comparing each coefficient of $d^i$ for $i \in \{0,1,2,3\}$, we have this inequality
 $$20736d^3-43776d^2+25344d-2304 > 20736d^3 - 47692.8d^2 + 20736d-20733.7$$
$$=20736(d^3-2.3d^2 +d-2.3)=\{20736(d^2+1)(d-2.3)\}>0$$
for $d \ge 3$.

{\it Case 4. $a=3$ and $b=3$}. By Observation~\ref{obs}, the only graph with $a=3$ and $b=3$ is $H_{d,1}$,
which completes the proof.

\section{Proof of Theorem~\ref{main2}}

If the edge-connectivity of a regular multigraph $G$ is even,
then there exists a vertex subset $S$ such that $|[S,V(G)-S]|$ is even.
Since $G$ is multigraph, the size of $S$ may be 2, so we have much simpler case than the ones of Theorem~\ref{main1}.

\medskip
\noindent
{\it Proof of Theorem~\ref{main2}}.
Assume to the contrary that $\kappa'(G) \le t$. Then there exists a vertex subset $S$ such that $l=|[S,\overline{S}]|\le t$
for some positive integer $l$.
Let $a=|A|$ and let $b=|\overline{A}|$, so we have $a+b=n$. We may assume that $a \ge b$.
Note that $a$ is at least 2. The quotient matrix of the partiton $S$ and $\overline{S}$ is 

$$A=\begin{pmatrix}
    d-\frac{l}{a} & \frac{l}{a} \\
    \frac{l}{b}  & d - \frac{l}{b}
\end{pmatrix}.$$

The characteristic polynomial of this matrix equals 
$$\left(x-d+\frac{l}{a}\right)\left(x-d+\frac{l}{b}\right)-\frac{l^2}{ab}$$,
$$=(x-d)^2+\left(\frac{l}{a}+\frac{l}{b}\right)(x-d)=(x-d)\left(x-d+\frac{l}{a}+\frac{l}{b}\right).$$
Thus the eigenvalues of the matrix $A$ are $d$ and $d-\frac{l}{a}-\frac{l}{b}$.

Since $b\ge a \ge 2$, by Theorem~\ref{GR1}, we have

\begin{equation}\label{keyequ} \lambda_2(G) \ge
d- \frac{l}{a}-\frac{l}{b} \ge d - \frac{l}{2} - \frac l2 = d- l \ge d-t. \end{equation}

Assume that $t$ is odd. Then we have $t \ge 3$.
If $l=t$, then since $b \ge a \ge 3$, 

\begin{equation}
\lambda_2(G) \ge
d- \frac{l}{a}-\frac{l}{b} \ge d-\frac{2l}3 = d-\frac{2t}3 \ge d-t+1. \end{equation}

If $l < t$, then since $b \ge a \ge 2$,
\begin{equation}
\lambda_2(G) \ge d- \frac{l}{a}-\frac{l}{b} \ge d- l \ge d-t+1. \end{equation}

\end{document}